\documentclass[final,5p,twocolumn,number]{elsarticle} 

\usepackage[fleqn]{amsmath}
\usepackage{amsfonts}
\usepackage{amssymb}
\usepackage{amsthm}
\usepackage{mathtools}
\usepackage{physics}
\usepackage{siunitx}
\usepackage{optidef}
\usepackage{eurosym}

\usepackage{adjustbox}
\usepackage{epstopdf}
\usepackage{float}
\usepackage{graphicx}
\usepackage{rotating}
\usepackage{subcaption}

\usepackage{tikz}
\usepackage{pgfplots}
\usepackage{pgfplotstable}
\pgfplotsset{compat=1.10}
\usepgfplotslibrary{fillbetween}
\usetikzlibrary{calc}
\usetikzlibrary{positioning}
\usetikzlibrary{decorations.markings}
\usetikzlibrary{trees}
\usetikzlibrary{patterns}
\usetikzlibrary{decorations.pathreplacing}
\usepackage{tikz-qtree}

\usepackage{booktabs}
\usepackage{ctable}
\usepackage{longtable}
\usepackage{threeparttable}
\usepackage{tabularx}
\usepackage{multirow}

\usepackage{xcolor}
\usepackage{color}
\usepackage[ruled]{algorithm2e}
\usepackage[labelfont={bf}]{caption}
\usepackage{enumerate}
\usepackage{lastpage}
\usepackage{listings}
\usepackage[colorlinks=true, allcolors=blue]{hyperref}
\usepackage[normalem]{ulem}

\usepackage{verbatim}
\usepackage{url}

\newcommand{\E}{\mathbb{E}}
\newcommand{\Prob}{\mathbb{P}}
\newcommand{\totvar}{|\Delta|}

\newcommand{\R}{\mathbb{R}}
\newcommand{\Z}{\mathbb{Z}}

\date{\today} 
\allowdisplaybreaks[1]

\newtheorem{theorem}{Theorem}

\newtheorem{lemma}{Lemma}

\theoremstyle{definition}
\newtheorem{definition}{Definition}

\newtheorem{assumption}{Assumption}

\journal{arXiv}

\begin{document}

\begin{frontmatter}

\title{Parametric error bounds for convex approximations of two-stage mixed-integer recourse models with a random second-stage cost vector}


\author[NTNU]{E. Ruben van Beesten\corref{correspondingauthor}}
\ead{ruben.van.beesten@ntnu.no}
\cortext[correspondingauthor]{Corresponding author}

\author[RuG]{W. Romeijnders}
\ead{w.romeijnders@rug.nl}

\address[NTNU]{Norwegian University of Science and Technology (NTNU), Department of Industrial Economics and Technology Management, NO-7491, Trondheim, Norway}
\address[RuG]{University of Groningen, Faculty of Economics and Business, Nettelbosje 2, 9747 AE, Groningen, the Netherlands}

\begin{abstract}
We consider two-stage recourse models in which the second-stage problem has integer decision variables and uncertainty in the second-stage cost vector, technology matrix, and the right-hand side vector. Such mixed-integer recourse models are typically non-convex and thus hard to solve. There exist convex approximations of these models with accompanying error bounds. However, it is unclear how these error bounds depend on the distributions of the second-stage cost vector $q$. In fact, the only error bound that is known hinges on the assumption that $q$ has a finite support. In this paper, we derive \textit{parametric} error bounds whose dependence on the distribution of $q$ is explicit and that hold for any distribution of $q$, provided it has a finite expected $\ell_1$-norm. We find that the error bounds scale linearly in the expected value of the $\ell_1$-norm of $q$.
\end{abstract}


\end{frontmatter}

\section{Introduction} \label{sec:introduction}

Two-stage mixed-integer recourse (MIR) models are a class of models that can be used to solve optimization problems under uncertainty. MIR models combine two computational difficulties: uncertainty of some of the model parameters and integer restrictions on some of the decision variables. These integer restrictions cause MIR models to be generally non-convex and hence, extremely hard to solve. Traditional solution methods for MIR models typically combine ideas from deterministic mixed-integer programming and stochastic continuous programming, see, e.g., \cite{Ahmed2004,Caroe1999,Laporte1993,ntaimo2010disjunctive,schultz1998solving,Sen2005}, and the survey papers \cite{kleinhaneveld1999,schultz2003stochastic,sen2005algorithms}. However, due to their reliance on non-convex optimization methods, these methods can have difficulty solving large-scale problems.

One alternative approach is to approximate the original non-convex MIR model by a convex model. Such a \textit{convex approximation} model can be solved efficiently using convex optimization techniques, thus overcoming the computational difficulties inherent in MIR models. The obvious drawback of this approach is that we only obtain an \textit{approximate} solution to the original MIR model, which may or may not be of good quality. Hence, performance guarantees are needed that ensure that the solution to the convex approximation model performs well for the original MIR model. 

In the literature, convex approximations with corresponding performance guarantees have been derived in the form of \textit{error bounds}: upper bounds on the approximation error \cite{romeijnders2016general,vanbeesten2020convex,vanderlaan2020LBDA}. These error bounds are small (and hence, the approximation is good) if the distribution of the second-stage right-hand-side vector is highly dispersed \cite{romeijnders2017assessing}. Analogous results are missing for the distribution of the second-stage cost vector, denoted $q$, however. In the literature, only non-parametric error bounds are known that implicitly depend on $q$. What is more, these error bounds are limited to the case where the support of $q$ is finite. In this paper, we derive \textit{parametric} error bounds that explicitly depend on (the distribution) of $q$ and that hold under mild assumptions on the distribution of $q$.

Mathematically, we consider two-stage MIR models of the form
\begin{align}
    \min_{x \in X} \{ c^\top x + \underbrace{\E^{\Prob}\big[v^q(h - Tx)\big]}_{Q(x)} \}, \label{eq:model}
\end{align}
where $x$ is the first-stage decision vector to be chosen from the feasible set $X \subseteq \R^{n_1}$, so as to minimize the sum of the first-stage costs $c^\top x$ and the expected second-stage costs. The second-stage costs are given by the value function $v^q(h - Tx)$, which depends on the first-stage decision $x$ and the value of the random variable $\xi := (q, T, h)$ with range $\Xi := \Xi^q \times \Xi^T \times \Xi^h$. For every $q \in \Xi^q$, the value function is defined as
\begin{align*}
    v^q(s) := \min_{y \in Y} \{ q^\top y \ | \ Wy = s \}, \quad s \in \R^m, 
\end{align*}
where $y$ is the second-stage decision vector to be chosen from the mixed-integer feasible set $Y := \Z_+^{n_2} \times \R_+^{\bar{n}_2}$ at a cost $q$, while satisfying the constraint $Wy = s$. The integer restrictions on $y$ cause the value function $v^q$, and consequently, also the recourse function $Q$, to be typically non-convex.

In the literature, several convex approximations $\tilde{Q}$ of $Q$ have been proposed \cite{kleinhaneveld2006simple,vlerk2004,vlerk2010}, and corresponding error bounds have been derived: upper bounds on the maximum approximation error $\| Q - \tilde{Q} \|_\infty := \sup_{x \in \R^{n_1}} | Q(x) - \tilde{Q}(x) |$ \cite{romeijnders2016general,romeijnders2015,romeijnders2016tu,vanbeesten2020convex,vanderlaan2020LBDA,vanderlaan2018higher}. However, the dependence of these error bounds on (the distribution of) $q$ has mainly been neglected in the literature. In most papers, $q$ is assumed to be fixed and non-parametric error bounds are derived that depend on $q$ implicitly. An exception exists for the special case of simple integer recourse (SIR) with $q$ deterministic, for which parametric error bounds are derived that scale linearly in the sum of the (assumed non-negative) elements of $q$ \cite{kleinhaneveld2006simple}. Another exception is the appendix of \cite{vanbeesten2020convex}, where non-parametric error bounds are derived for general MIR models with a \textit{random} $q$. These results are quite limited, though, as they only hold under the assumption that $q$ is discretely distributed with a finite support.

We contribute to this literature in two ways. First, we derive parametric error bounds on $\| Q - \tilde{Q} \|_\infty$ under the assumption that $q$ is fixed. We find that these error bounds scale linearly in the $\ell_1$ norm of $q$. Hence, this result can be seen as a generalization of the bounds from \cite{kleinhaneveld2006simple} for SIR models to a much more general setting. Second, we use this result to derive parametric error bounds on $\| Q - \tilde{Q} \|_\infty$ for the case that $q$ is random. It turns out that these error bounds only depend on $q$ through $\E^{\Prob}\big[ \| q \|_1 \big]$, the expected value of the $\ell_1$-norm of $q$, and that the bounds scale linearly in $\E^{\Prob}\big[ \| q \|_1 \big]$. Hence, only the average ``magnitude'' of $q$ is relevant for the error bounds. In particular, in contrast with the distribution of $h$, the dispersion of the distribution of $q$ turns out to be completely irrelevant.

Throughout the paper, we make the following general assumptions. 
\begin{assumption} \label{ass:main}
    We assume that
    \begin{enumerate}[(a)]
        \item for every $q \in \Xi^q$, the recourse is complete and sufficiently expensive, i.e., $-\infty < v^q(s) < + \infty$, for all $s \in \R^m$, \label{ass:complete_suff_exp_recourse}
        \item the expectation of the $\ell_1$ norm of $h$ and $q$ are finite, i.e., $\E^{\Prob}\big[ \| h \|_1 \big] < + \infty$ and $\E^{\Prob}\big[ \| q \|_1 \big] < +\infty$, \label{ass:finite_exp}
        \item $h$ is continuously distributed with joint pdf $f$, and $(q,T)$ and $h$ are pairwise independent, and \label{ass:distributional}
        \item the recourse matrix $W$ is integer. \label{ass:integer_W}
    \end{enumerate}
\end{assumption} 
Assumption~\ref{ass:main}\eqref{ass:complete_suff_exp_recourse}-\eqref{ass:finite_exp} guarantee that the recourse function $Q(x)$ is finite for every $x \in \R^{n_1}$. Assumption~\ref{ass:main}\eqref{ass:distributional} restricts the distribution right-hand side vector $h$ to continuous distributions only. This is in line with the literature and crucial for the total variation-based error bounds that we will derive. The independence assumption in Assumption~\ref{ass:main}\eqref{ass:distributional} is for ease of presentation; similar results can be derived under relaxed versions of this assumption. However, at the very least, full dependence between $(q,T)$ and $h$ should be ruled out. Finally, Assumption~\ref{ass:main}\eqref{ass:integer_W} is required for the derivation of our error bounds in Section~\ref{sec:error_bounds}. However, it is not very restrictive, as any rational matrix can be transformed into an integer matrix by appropriate scaling. 

The remainder of this paper is structured as followed. In Section~\ref{sec:problem_definition} we provide a detailed problem definition. We define two convex approximations and we discuss the main difficulty in deriving error bounds for these approximations. In Section~\ref{sec:properties_of_approx_error} we derive two properties of the value function approximation error: asymptotic periodicity and a uniform upper bound. We use these properties in Section~\ref{sec:error_bounds} to first derive an error bound on $\| Q - \tilde{Q} \|_\infty$ under the assumption that $q$ is fixed, which explicitly depends on $q$. Then, we use this result to derive an error bound when $q$ is random. Finally, Section~\ref{sec:conclusion} concludes the paper.

\section{Problem definition} \label{sec:problem_definition}

In this section we provide a detailed problem definition. First, in Section~\ref{subsec:convex_approximations} we define two convex approximations from the literature. Second, in Section~\ref{subsec:error_boundss} we review the most general (non-parametric) error bounds from the literature and we discuss why extending these to parametric error bounds in (the distribution of) $q$ is non-trivial. This motivates our analysis in the subsequent sections.

\subsection{Convex approximations} \label{subsec:convex_approximations}

In the literature, several authors have developed convex approximations of MIR models \cite{kleinhaneveld2006simple,romeijnders2016general,romeijnders2015,romeijnders2016tu,vanderlaan2020LBDA,vlerk2004,vlerk2010}. From these, we consider two particular convex approximations: the \textit{shifted LP-relaxation approximation} \cite{romeijnders2016general} and the \textit{$\alpha$-approximation} \cite{vanderlaan2020LBDA}. In this subsection, we straightforwardly extend these definitions to our setting where $q$ has an arbitrary support.

The starting point for defining both convex approximations is the dual representation of the LP-relaxation of the value function $v^q$, given by
\begin{align*}
    v^q_{\text{LP}}(s) &= \max_{\lambda \in \R^m} \{ \lambda^\top s \ | \ \lambda^\top W \leq q \} \\
    &= \max_{k \in K^q} \{ (\lambda^q_k)^\top s \}, \qquad\qquad\qquad s \in \R^m,
\end{align*}
where $\lambda^q_k$, $k \in K^q$, are the vertices of the dual feasible region $\{ \lambda \in \R^m \ | \ \lambda^\top W \leq q \}$, defined as $\lambda^q_k := q_{B^k}^\top(B^k)^{-1}$, where $B^k$ is the corresponding dual feasible basis matrix, and $q_{B^k}$ is the vector of elements of $q$ that correspond to the basis matrix $B^k$, $k \in K^q$. The index set $K^q$ denotes which basis matrices are dual feasible for $q \in \Xi^q$. Using a result from the literature (see Theorem~2.9 from \cite{romeijnders2016general}), we can partition the domain $\R^m$ of $v^q_{\text{LP}}$ into cones $\Lambda^k := \{ s \in \R^m \ | \ (B^k)^{-1} s \geq 0 \}$, $k \in K^q$, such that the dual vertex $\lambda^q_k$ is optimal whenever $s \in \Lambda^k$, i.e.,
\begin{align*}
    v^q_{\text{LP}}(s) = (\lambda^q_k)^\top s, \quad s \in \Lambda^k.
\end{align*}

We can derive a similar partial representation for $v^q$ itself. First write $v^q(s) = v^q_{\text{LP}}(s) + \psi^q(s)$, so $\psi^q$ represents the ``cost of the integer restrictions''. Using Gomory relexations of the value function $v^q$, Romeijnders et al. \cite{romeijnders2016general} show that $\psi^q$ is $B^k$-periodic on shifted versions of the cones $\Lambda^k$, $k \in K^q$.

\begin{definition}\label{def:B-periodicity}
    Let $f: \R^m \to \R$ and $B \in \R^{m \times m}$ be given. Then, $f$ is \textit{$B$-periodic} if $f(s) = f(s + B \ell)$ for every $s \in \R^m$ and $\ell \in \Z^m$. 
\end{definition}

\begin{lemma} \label{lemma:psi}
    Let $q \in \Xi^q$ be given and consider the value function $v^q$ and its LP-relaxation $v^q_{\text{LP}}$. Then, there exist constants $d^k > 0$, $k \in K^q$, such that 
    \begin{align*}
        v^q(s) = (\lambda^q_k)^\top s + \psi^q_k(s), \quad s \in \Lambda^k(d^k),
    \end{align*}
    where $\psi^q_k$ is a $B^k$-periodic function and $\Lambda^k(d^k) := \{ s \in \R^m \ | \ \mathcal{B}(s,d^k) \subseteq \Lambda^k \}$, where $\mathcal{B}(s,d^k) := \{ t \in \R^m \ | \ \| t - s \|_2 \leq d^k \}$ is the closed ball of radius $d^k$ centered at $s$.
\end{lemma}
\begin{proof}
    See Theorem~2.9 in \cite{romeijnders2016general}.
\end{proof}

From Lemma~\ref{lemma:psi} we learn that, at least on the shifted cones $\Lambda^k(d^k)$, $k \in K^q$, convexity of $v^q$ is destroyed by the periodicity of the function $\psi^q_k$. This observation has led to the proposal of two convex approximations of $v^q$, based on ``convexifying'' adjustments of the functions $\psi^q_k$, $k \in K^q$: the \textit{shifted LP-relaxation approximation} and the \textit{$\alpha$-approximation}.

The shifted LP-relaxation approximation $\hat{v}^q$ is constructed by replacing the periodic function $\psi^q_k$ by its mean value $\Gamma^q_k$. Since every $B^k$-periodic function is also $p_k I_m$-periodic with $p_k := |\det(B^k)|$ (see \cite{romeijnders2016general}), this mean value can be defined as
\begin{align}
	\Gamma^q_k := p_k^{-m} \int_0^{p_k} \cdots \int_{0}^{p_k} \psi^q_k(s) ds_1 \cdots ds_m. \label{eq:def_Gamma_k}
\end{align}	
Taking the maximum over all $k \in K^q$ then yields the approximation.

\begin{definition} \label{def:shifted_LP-relaxation}
    Consider the value function $v^q$ for a given value of $q \in \Xi^q$. Then, its \textit{shifted LP-relaxation approximation} is given by
    \begin{align*}
        \hat{v}^q(s) &:= \max_{k \in K^q} \big\{ (\lambda^q_k)^\top s + \Gamma^q_k \big\}, \quad s \in \R^m,
    \end{align*}
    where $\Gamma^q_k$ is the mean value from \eqref{eq:def_Gamma_k}. The corresponding shifted LP-relaxation approximation of the recourse function $Q$ is defined as
    \begin{align*}
        \hat{Q}(x) := \E^{\Prob}\big[ \hat{v}^q(h - Tx) \big], \quad x \in \R^{n_1}.
    \end{align*}
\end{definition}

To construct the $\alpha$-approximation $\tilde{v}_\alpha$ of $v^q$, we replace the term $Tx$ in $\psi^q_k(h - Tx)$ by a constant vector $\alpha \in \R^m$, yielding $\psi^q_k(h - \alpha)$. Note that we treat $h$ and $Tx$ differently, hence the $\alpha$-approximation is defined as a function of $h$ and $Tx$ separately instead of $s = h - Tx$. Again, we take the maximum over all $k \in K^q$ to yield our approximation.

\begin{definition} \label{def:alpha-approximation}
    Consider the value function $v^q$ for a given value of $q \in \Xi^q$. Then, its \textit{$\alpha$-approximation} is given by
    \begin{align*}
        \tilde{v}_\alpha^q(h,Tx) &:= \max_{k \in K^q} \big\{ (\lambda^q_k)^\top (h - Tx) + \psi^q_k(h - \alpha) \big\}, 
    \end{align*}
    $h \in \R^m$, $Tx \in \R^m$. The corresponding $\alpha$-approximation of the recourse function $Q$ is defined as
    \begin{align*}
        \tilde{Q}_\alpha(x) := \E^{\Prob}\big[ \tilde{v}_\alpha^q(h,Tx) \big], \quad x \in \R^{n_1}.
    \end{align*}
\end{definition}

\subsection{Error bounds} \label{subsec:error_boundss}
In this subsection we discuss non-parametric error bounds from the literature and we discuss why extending them to error bounds that are parametric in (the distribution of) $q$ is not trivial. In our discussion we focus on the shifted LP-relaxation approximation; the analysis for the $\alpha$-approximation is completely analogous. 

As a starting point, we take a result from the literature that provides a non-parametric error bound under the assumption that $q$ and $T$ are fixed.  Consider the shifted LP-relaxation approximation $\hat{Q}$ from Definition~\ref{def:shifted_LP-relaxation} under the assumption that $q$ and $T$ are fixed. Romeijnders et al. \cite{romeijnders2016general} derive an error bound for this setting. We restate the result here after providing two definitions.

\begin{definition}\label{def:total_variation}
    Let $f: \mathbb{R}  \to \mathbb{R}$ be a real-valued function and let $I \subset \mathbb{R}$ be an interval. Let $\Pi(I)$ denote the set of all finite ordered sets $P = \{z_1, \ldots, z_{N+1} \}$ with $z_1 < \cdots < z_{N+1}$ in $I$. Then, the \emph{total variation} of $f$ on $I$, denoted by $|\Delta|f(I)$, is defined by
	\begin{align*}
	    | \Delta | f(I) := \sup_{P \in \Pi(I)} V_f(P),
	\end{align*}
	where $V_f(P) := \sum_{i = 1}^{N} | f(z_{i+1}) - f(z_i) |$. We write $|\Delta|f := |\Delta|f(\mathbb{R})$. We say that $f$ is of \emph{bounded variation} if $|\Delta|f < +\infty$.
\end{definition}

\begin{definition} \label{def:H^m}
	We denote by $\mathcal{H}^m$ the set of all $m$-dimensional joint pdfs $f$ whose one-dimensional conditional density functions $f_i(\cdot|t_{-i})$ are of bounded variation for all $t_{-i} \in \R^{m-1}$, $i=1,\ldots,m$.
\end{definition}

\begin{lemma}\label{lemma:error_bound_literature}
    Consider the recourse function $Q$ and its shifted LP-relaxation approximation $\hat{Q}$ from Definition~\ref{def:shifted_LP-relaxation} and assume that $q \in \Xi^q$ and $T \in \Xi^T$ are fixed. Then, there exists a finite constant $\tilde{C} > 0$, not depending on $T$, such that for all $f \in \mathcal{H}^m$, we have
    \begin{align*}
        \| Q - \hat{Q} \|_\infty &\leq \tilde{C} \sum_{i=1}^m \E^{h_{-i}}\big[ \totvar f_i(\cdot | h_{-i}) \big], \quad x \in \R^{n_1}.
    \end{align*}
\end{lemma}
\begin{proof}
    See Theorem~5.1 in \cite{romeijnders2016general}.
\end{proof}
 
Observe that the constant $\tilde{C}$ depends on $q$, but the dependence structure is not made explicit in the lemma above. Only existence of some constant $\tilde{C}$ is proven. Moreover, we will show that as a result, extending the error bound above to a setting where $q$ is stochastic with an infinite support is non-trivial.

Suppose that $q$ and $T$ are stochastic and that Assumption~\ref{ass:main} holds. Our aim is to find an upper bound on the maximum approximation error $\| Q - \hat{Q} \|_\infty$, i.e., a uniform upper bound on $| Q(x)  - \hat{Q}(x) |$ over all $x \in \R^{n_1}$. By definition, we have for all $x \in \R^{n_1}$,
\begin{align}
    &| Q(x) - \hat{Q}(x) | = \big| \E^{\xi}\big[ v^q(h - Tx) - \hat{v}^q(h - Tx) \big] \big| \nonumber\\
    &\qquad\leq \E\bigg[ \Big| \E^{h|q,T} \big[ v^q(h - Tx) - \hat{v}^q(h - Tx) \big] \Big| \bigg], \label{eq:approx_error_q_t_h}
\end{align}
where we use Jensen's inequality and the fact that $(q,T)$ and $h$ are mutually independent. Applying Lemma~\ref{lemma:error_bound_literature} to the inner expression $\big| \E^{h|q,T} \big[ v^q(h - Tx) - \hat{v}^q(h - Tx) \big] \big|$ yields, for every $x \in \R^{n_1}$,
\begin{align}
    |Q(x) - \hat{Q}(x)| &\leq \E\big[ \tilde{C}^q \big] \sum_{i=1}^m \E^{h_{-i}}\big[ \totvar f_i(\cdot | h_{-i}) \big], \label{eq:error_bound_infinite}
\end{align}
where $\tilde{C}^q$ is the constant from Lemma~\ref{lemma:error_bound_literature} corresponding to $q \in \Xi^q$. If $q$ has a finite support, then the fact that $\tilde{C}^q$ is finite for every $q \in \Xi^q$ guarantees that the expected value $\E\big[ \tilde{C}^q \big]$ is also finite. This is indeed the approach taken in \cite{vanbeesten2020convex}. This assumption can be quite restrictive, though, as in reality, cost coefficients might be appropriately modeled by, e.g., continuous random variables. If we relax the assumption, however, we cannot guarantee that $\E\big[ \tilde{C}^q \big]$ is finite. Hence, in order to derive finite error bounds that hold for more general distributions of $q$, we need to further investigate the dependence of $\tilde{C}^q$ on $q$.

Although Lemma~\ref{lemma:error_bound_literature} merely claims the existence of some constant $\tilde{C}^q>0$, its proof in \cite{romeijnders2016general} is constructive, i.e., it finds a particular value for $\tilde{C}^q$. However, due to the particular way that $\tilde{C}^q$ is constructed, it turns out that analyzing the dependence structure between $\tilde{C}^q$ and $q$ is extremely difficult. For this reason, we take the following alternative route. First, we derive an alternative to Lemma~\ref{lemma:error_bound_literature}, with an alternative constant $C^q > 0$, whose dependence on $q$ can be expressed explicitly. Then, using this alternative result, we derive an analogue to the error bound \eqref{eq:error_bound_infinite}, which explicitly depends on the distribution of $q$ and which we \textit{can} guarantee to be finite.

\section{Properties of the value function approximation error} \label{sec:properties_of_approx_error}

In this section we derive two properties of the approximation error $\hat{v}^q - v^q$ of the shifted LP-relaxation approximation that will be used to derive our error bounds: asymptotic periodicity and a uniform error bound. Analogous results can be derived for the approximation error $\tilde{v}^q_\alpha - v^q$ of the $\alpha$-approximation; for the sake of brevity we skip the analysis.

\subsection{Asymptotic periodicity} \label{subsec:asymptotic_periodicity}

Consider the shifted LP-relaxation $\hat{v}^q$ from Definition~\ref{def:shifted_LP-relaxation} for a given $q \in \Xi^q$. We will prove that the corresponding approximation error $v^q - \hat{v}^q$ is asymptotically periodic, i.e., we show that on a ``relatively large'' part of its domain, the function $v^q - \hat{v}^q$ is a periodic function. Specifically, we prove that there exist vectors $\bar{\sigma}^k$, $k \in K^q$, such that $v^q - \hat{v}^q$ is $B^k$-periodic on the shifted cone $\bar{\sigma}^k + \Lambda^k$, $k \in K^q$.

In fact, such an asymptotic periodicity result has already been proven in Proposition~3.7 in \cite{romeijnders2016general}. However, in their result, the vector corresponding to $\bar{\sigma}^k$ depends on $q$. Since $q$ is fixed in their paper, this does not hinder their analysis. However, in our setting with a random $q$, this is a crucial obstacle to deriving an asymptotic error bound. We will highlight why this is the case when we derive our error bound in Section~\ref{sec:error_bounds}. Hence, in this section we aim at vectors $\bar{\sigma}^k$ that do not depend on $q$. One complicating factor here is that the index $k$ is taken from a set $K^q$, which depends on $q$. To avoid confusion, we define the set $\bar{K} := \cup_{q \in \Xi^q} := K^q$ of all indices $k$ for which the basis matrix $B^k$ is dual feasible for some $q \in \Xi^q$. Note that $\bar{K}$ is a finite set because $W$ only has a finite number of basis matrices. We will derive a vector $\bar{\sigma}^k$ for every $k \in \bar{K}$. In our analysis we will focus on the shifted LP-relaxation $\hat{v}^q$. The analysis for the $\alpha$-approximation is analogous.

Let $k \in \bar{K}$ and $q \in \Xi^q$ with $k \in K^q$ be given. Then, by Lemma~\ref{lemma:psi} we know that
$v^q(s) = (\lambda^q_k)^\top s + \psi^q_k(s)$ whenever $s \in \Lambda^k(d^k)$. If we can find a vector $\bar{\sigma}^k \in \Lambda^k(d^k)$ such that 
\begin{align}
    \hat{v}^q(s) = (\lambda^q_k)^\top s + \Gamma^q_k, \label{eq:v_hat_opt_k}
\end{align}
for $s \in \bar{\sigma}^k + \Lambda^k$, then because $\bar{\sigma}^k + \Lambda^k \subseteq \Lambda^k(d^k)$, it follows that 
\begin{align}
    v^q(s) - \hat{v}^q(s) = \psi^q_k(s) - \Gamma^q_k, \label{eq:v_approx_error_periodic}
\end{align}
for all $s \in \bar{\sigma}^k + \Lambda^k$. Hence, the approximation error $v^q - \hat{v}^q$ is $B^k$-periodic on the shifted cone $\bar{\sigma}^k + \Lambda^k$ with a mean value of zero (since $\Gamma^q_k$ is the mean value of $\psi^q_k$). It remains to find a vector $\bar{\sigma}^k \in \Lambda^k(d^k)$ that satisfies \eqref{eq:v_hat_opt_k}.

By definition of $\hat{v}^q$, equation \eqref{eq:v_hat_opt_k} is equivalent to the statement that for every $l \in K^q$, we have
\begin{align}
   (\lambda^q_k - \lambda^q_l)^\top s  \geq \Gamma^q_l -  \Gamma^q_k. \label{eq:lambda_Gamma_inequality}
\end{align}
We analyze the left-hand side and right-hand side of the inequality above separately. For the left-hand side we have the following representation.

\begin{lemma} \label{lemma:lambda_diff_B}
	Let $k,l \in \bar{K}$ be given. Define $N^l$ as the matrix consisting of the columns of $W$ that are not columns in the basis matrix $B^k$, and write $q_{N^l}$ for the vector of elements of $q$ corresponding to the columns of $N^l$. For $i=1,\ldots,m$, write $B^k_i \in \mathcal{B}^l$ if the $i$th column of $B^k$ is also a column in $B^l$, corresponding to the same second-stage variable. If $B^k_i \notin \mathcal{B}^l$, then write $j(i)$ for the index of $N^l$ such that $B^k_i = N^l_{j(i)}$, where both columns correspond to the same second-stage variable. Then,	for every $q \in \Xi^q$ for which $k,l \in K^q$, we have
	\begin{align*}
    	\bar{q}^{kl}_i &:= (\lambda^q_k - \lambda^q_l)^\top B^k_i = \begin{cases}
        	0, & \text{if } B^k_i \in \mathcal{B}^l, \\
        	\bar{q}_{N^l_{j(i)}}, &\text{if } B^k_i \notin \mathcal{B}^l,
    	\end{cases}		
	\end{align*}
	where $(\bar{q}_{N^l})^\top := (q_{N^l})^\top - (q_{B^l})^\top(B^l)^{-1}N^l$ denotes the reduced cost of $y_{N^l}$.
\end{lemma}
\begin{proof}
	Let $q \in \Xi^q$ with $k,l \in K^q$ and $i=1,\ldots,m$ be given. Then, $(\lambda^q_k)^\top B^k_i = (q_{B^k})^\top(B^k)^{-1} B^k_i = (q_{B^k})^\top e_i = q_{B^k_i}$, where $e_i \in \R^m$ is the $i$th unit vector. Next, consider $(\lambda^q_l)^\top B^k_i$. If $B^k_i \in \mathcal{B}^l$, then, writing $B^k_i = B^l_{r(i)}$, we have $(\lambda^q_l)^\top B^k_i = (q_{B^l})^\top(B^l)^{-1} B^l_{r(i)} = (q_{B^l})^\top_{r(i)} = q_{B^l_{r(i)}} = q_{B^k_i}$. Conversely, if $B^k_i \notin \mathcal{B}^l$, then we have $(\lambda^q_l)^\top B^k_i = (q_{B^l})^\top(B^l)^{-1} N^l_{j(i)} = q_{N^l_{j(i)}} - \bar{q}_{N^l_{j(i)}} = q_{B^k_i} - \bar{q}_{N^l_{j(i)}}$. Combining these findings yields the result.
\end{proof}

Next, we consider the right-hand side of \eqref{eq:lambda_Gamma_inequality}. We first derive an upper bound $G^q_{kl}$ on the difference $\Gamma^q_l -  \Gamma^q_k$.

\begin{lemma}\label{lemma:G}
	Let $k,l \in \bar{K}$ be given. Then, there exists $t^{kl} \in \R^m_+$, such that for all $q \in \Xi^q$ with $k,l \in K^q$, we have
	\begin{align*}
		\Gamma^q_l - \Gamma^q_k \leq G^q_{kl} := (\bar{q}^{kl})^\top t^{kl}.
	\end{align*}
\end{lemma}
\begin{proof}
	Let $q \in \Xi^q$ with $k,l \in K^q$ be given. Writing $v^q_{B^k}$ for the Gomory relaxation of $v^q$ with respect to the basis matrix $B^k$, we know from Theorem~2.9 in \cite{romeijnders2016general} that $v^q_{B^k}(s) = v^q(s) \geq v^q_{B^l}(s)$ for all $s \in \Lambda^k(d^k)$. By Lemma~2.3 in \cite{romeijnders2016general} it follows that we can represent the Gomory relaxations of $v^q$ as $v_{B^{\bar{k}}}(s) = \lambda^q_{\bar{k}} s + \psi^q_{\bar{k}}(s)$, $s \in \R^m$, $\bar{k} \in K^q$. It follows that for all $s \in \Lambda^k(d^k)$ we have $\psi^q_l(s) - \psi^q_k(s) \leq  (\lambda^q_k -\lambda^q_l)^\top s$.	Note that $\psi^q_k$ and $\psi^q_l$ are $B^k$-periodic and $B^l$-periodic, respectively. By Lemma~4.8 in \cite{romeijnders2016general} this implies that they are $p_k I_m$-periodic and $p_l I_m$-periodic, respectively, where $p_k := | \det B^k |$ and $p_l := | \det B^l |$. Note that $p_k$ and $p_l$ are integers by our assumption that $W$ is an integer matrix. It follows that $\psi^q_l - \psi^q_k$ is a $p_{kl} I_m$-periodic function, where $p_{kl} := p_k \cdot p_l$. Now, let $C_{kl} \subseteq \Lambda^k(d^k)$ be a hypercube of length $p_{kl}$. Then, integrating $\psi^q_l - \psi^q_k$ over $C_{kl}$ and dividing by its volume $(p_{kl})^m$, we obtain 
	\begin{align*}
    	\Gamma^q_l - \Gamma^q_k &= p_{kl}^{-m} \int_{C_{kl}} (\psi^q_l(s) - \psi^q_k(s)) ds \\
    	&\leq p_{kl}^{-m} \int_{C_{kl}} (\lambda^q_k - \lambda^q_l)^\top s ds =: \tilde{G}^q_{kl}.
	\end{align*}
	We will derive an upper bound $G^q_{kl}$ on the right-hand side $\tilde{G}^q_{kl}$. Using the change of variables $s = B^k t$, we can write
	\begin{align*}
    	\tilde{G}^q_{kl} &= p_{kl}^{-m} | \det B^k | \int_{\bar{C}_{kl}} (\lambda^q_k - \lambda^q_l)^\top B^k t dt \\
    	&= p_{kl}^{-m} | \det B^k | \int_{\bar{C}_{kl}} (\bar{q}^{kl})^\top t dt,
	\end{align*}
	where $\bar{C}_{kl} := \{ t \in \R^m_+ \ | \ B^k t \in C_{kl} \}$ and $\bar{q}^{kl}$ is as in Lemma~\ref{lemma:lambda_diff_B}. We claim that $\bar{q}^{kl}_i \geq 0$ for every $i =1,\ldots,m$. If $B^k_i \in B^l$, this follows immediately from Lemma~\ref{lemma:lambda_diff_B}. If $B^k_i \notin B^l$, then by Lemma~\ref{lemma:lambda_diff_B}, $\bar{q}^{kl}_i$ equals $\bar{q}_{N^l_{j(i)}}$, the reduced cost of the variable corresponding to $B^k_i$ with respect to the basis matrix $B^l$. Since $B^l$ is an dual feasible basis matrix, the reduced cost $\bar{q}_{N^l_{j(i)}}$ is non-negative. Hence, indeed $\bar{q}^{kl} \geq 0$. Define the vector $t^{kl}$ with elements $t^{kl}_i := \max\{ t_i \ | \ t \in \bar{C}_{kl} \}$, $i=1,\ldots,m$. Then, it follows that
	\begin{align*}
    	\tilde{G}^q_{kl} &\leq p_{kl}^{-m} | \det B^k | \int_{\bar{C}_{kl}} (\bar{q}^{kl})^\top t^{kl} dt \\
    	&= p_{kl}^{-m} \int_{C_{kl}} (\bar{q}^{kl})^\top t^{kl} ds = (\bar{q}^{kl})^\top t^{kl} = G^q_{kl}.
	\end{align*}	
	We conclude that $\Gamma^q_l - \Gamma^q_k \leq \tilde{G}^q_{kl} \leq G^q_{kl}$.
\end{proof}

By Lemma~\ref{lemma:G}, $(\lambda^q_k - \lambda^q_l)^\top s \geq G^q_{kl}$ is a sufficient condition for \eqref{eq:lambda_Gamma_inequality}. We use this fact to derive an vector $\bar{\sigma}^k$ for which \eqref{eq:v_approx_error_periodic} holds.

\begin{lemma} \label{lemma:sigma_bar_k}
	Let $k \in \bar{K}$ be given. Then, there exists $\bar{\sigma}^k \in \Lambda^k(d^k)$, such that for all $q \in \Xi^q$ with $k \in K^q$ and for all $s \in \bar{\sigma}^k + \Lambda^k$, we have
	\begin{align}
		v^q(s) - \hat{v}^q(s) = \psi^q_k(s) - \Gamma^q_k, \label{eq:v_approx_error_periodic_lemma}
	\end{align}
\end{lemma}
\begin{proof}
	Let $s \in \Lambda^k$ be given. Then, there exists some $t \in \R^m_+$ such that $s = B^k t$. Hence, for any $l \in \bar{K}$ with $l \neq k$ and any $q \in \Xi^q$ with $k,l \in K^q$, we can write
	\begin{align}
    	&(\lambda^q_k -\lambda^q_l)^\top s \geq G^q_{kl} \quad \iff \quad (\lambda^q_k -\lambda^q_l)^\top B^k t \geq G^q_{kl}\nonumber \\
    	&\qquad\qquad \iff \quad (\bar{q}^{kl})^\top t \geq (\bar{q}^{kl})^\top t^{kl}. \label{eq:ineq_lambda_G}
	\end{align}	
	Since $\bar{q}^{kl} \geq 0$ by the proof of Lemma~\ref{lemma:G}, a sufficient condition for \eqref{eq:ineq_lambda_G} is $t \geq t^{kl}$, which is equivalent to $s \in \Lambda^{kl} := \{ B^k t \ | \ t \geq t^{kl} \}$. Now, similar as in \cite{romeijnders2016general} it can be shown that the intersection $\bar{\Lambda}^k := \bigcap_{K^q : q \in \Xi^q} \bigcap_{l \in K^q : l \neq k} \Lambda^{kl}$ can be represented as $\bar{\sigma}^k + \Lambda^k$, for some $\bar{\sigma}^k \in \Lambda^k$. Note that here, the first intersection is over a \textit{finite} collection of index sets $K^q$, $q \in \Xi^q$, since $K^q \subseteq \bar{K}$ for every $q \in \Xi^q$ and $\bar{K}$ is a finite set. By construction of $\bar{\sigma}^k$ and $t^{kl}$, we have $\bar{\sigma}^k \in \Lambda^k(d^k)$. It then follows from the discussion at the start of this subsection that indeed, \eqref{eq:v_approx_error_periodic_lemma} holds if $s \in \bar{\sigma}^k + \Lambda^k$. 
\end{proof}

\subsection{Uniform upper bound} \label{subsec:value_function_approx_bound}
Next, we derive a uniform upper bound on the value function approximation error $\|\hat{v}^q - v^q\|_\infty$, whose dependence on $q$ is expressed explicitly. In particular, we derive a bound of the form $\|\hat{v}^q - v^q\|_\infty \leq \gamma \| q \|_1$, for some $\gamma > 0$. To derive such an upper bound, we split up the approximation error by the inequality
\begin{align}
    \| v^q - \hat{v}^q \|_\infty &\leq \| v^q - v_{\text{LP}}^q \|_\infty + \| v_{\text{LP}}^q - \hat{v}^q \|_\infty, \label{eq:v_v_hat_error_split}
\end{align}
and we bound each of the terms in the right-hand side above separately.

\begin{lemma}\label{lemma:bound_v_v_LP}
    There exists a finite constant $\gamma_1 > 0$, such that for every $q \in \Xi^q$,
    \begin{align*}
        \| v^q - v_{\text{LP}}^q \|_\infty \leq \gamma_1 \| q \|_1.
    \end{align*}
\end{lemma}
\begin{proof}
    See Corollary~2 in \cite{cook1986} and Remark~1 in the same paper.
\end{proof}

\begin{lemma}\label{lemma:bound_v_LP_v_hat}
    There exists a finite constant $\gamma_2 > 0$, such that for every $q \in\Xi^q$,
    \begin{align*}
        \| v_{\text{LP}}^q - \hat{v}^q \|_\infty \leq \gamma_2 \| q \|_1
    \end{align*}
\end{lemma}
\begin{proof}
    Comparing the dual formulation of $v_{\text{LP}}^q(s)$ with the definition of $\hat{v}^q(s)$, it is clear that $\|v_{\text{LP}}^q - \hat{v}^q \|_\infty \leq \max_{k \in K^q} \Gamma_k^q$. Recall that $\Gamma_k^q$ is the mean value of the $B^k$-periodic function $\psi_k^q$. By the proof of Theorem~2.9 in \cite{romeijnders2016general}, we can write $\psi^q_k(s) = \bar{q}_{N^k}^\top y^*_{N^k}$, where $\bar{q}_{N^k}^\top := q_{N^k}^\top - q_{B^k}^\top(B^k)^{-1} N^k \geq 0$, and $y^*_{N^k} \in [0, p_k]^m$ is optimal in the Gomory relaxation $v_{B^k}(s)$. Note that $\bar{q}_{N^k}^\top = [q_{N^k}^\top \ q_{B^k}^\top] \begin{bmatrix} I_{\bar{n}} \\ -(B^k)^{-1} N^k \end{bmatrix}$, where $\bar{n} := n_2 + \bar{n}_2 - m$. Hence, there exists a matrix $M^k$ (whose columns are a permutation of the columns of the matrix above) such that we can write $\bar{q}_{N^k}^\top = q^\top M^k$. It follows that
    \begin{align*}
        \Gamma_k^q &\leq \sup_{s \in \R^m} \psi^q_k(s) \leq \sup_{s \in \R^m} \bar{q}_{N^k}^\top y^*_{N^k}(s) \\
        &= \sup_{s \in \R^m} q^\top M^k y^*_{N^k}(s) \leq q^\top \big( p^k \cdot M^k \iota_{\bar{n}}\big),
    \end{align*}
    where $\iota_{\bar{n}} := (1,\ldots,1) \in \R^{\bar{n}}$. Hence, we obtain
    \begin{align*}
        &\|v_{\text{LP}}^q - \hat{v}^q \|_\infty \leq \max_{k \in K^q} \Gamma_k^q \leq \max_{k \in K^q} q^\top \big( p^k \cdot M^k \iota_{\bar{n}}\big) \\
        &\leq \max_{k \in \bar{K}} \left\{ p^k  q^\top M^k \iota_{\bar{n}}\right\} \leq \max_{k \in \bar{K}} \left\{ | p^k| \cdot \| q \|_1 \cdot \| M^k \iota_{\bar{n}} \|_\infty \right\}\\
        &= \max_{k \in \bar{K}} \left\{ | p^k| \cdot \| M^k \iota_{\bar{n}} \|_\infty \right\} \cdot \| q \|_1.
    \end{align*}
    Defining $\gamma_2 := \max_{k \in \bar{K}} \left\{ | p^k| \cdot \| M^k \iota_{\bar{n}} \|_\infty \right\}$, it follows that $\|v_{\text{LP}}^q - \hat{v}^q \|_\infty \leq \gamma_2 \| q \|_1$. Since $\bar{K}$ is a finite index set, $\gamma_2$ is indeed finite.
\end{proof}

Combining Lemma~\ref{lemma:bound_v_v_LP} and \ref{lemma:bound_v_LP_v_hat} yields desired upper bound on the approximation error $\| v^q - \hat{v}^q \|_\infty$.

\begin{lemma} \label{lemma:bound_v_v_hat}
    There exists a finite constant $\gamma > 0$ such that for all $q \in \Xi^q$,
    \begin{align*}
        \| v^q - \hat{v}^q \|_\infty &\leq \gamma \| q \|_1.
    \end{align*}
\end{lemma}
\begin{proof}
    Follows immediately from \eqref{eq:v_v_hat_error_split}, Lemma~\ref{lemma:bound_v_v_LP}, and Lemma~\ref{lemma:bound_v_LP_v_hat}.
\end{proof}

\section{Parametric error bounds} \label{sec:error_bounds}

In this section we derive parametric error bounds for the convex approximations of MIR models studied in this paper. As discussed in Section~\ref{sec:problem_definition}, we first derive an alternative to Lemma~\ref{lemma:error_bound_literature} for which the dependence of the corresponding constant $C^q$ on $q$ is explicit. In our discussion we will focus on the shifted LP-relaxation approximation. The analysis for the $\alpha$-approximation is analogous.

Consider the shifted LP-relaxation approximation $\hat{Q}$ from Definition~\ref{def:shifted_LP-relaxation} and suppose that $q$ and $T$ are fixed. To derive our alternative to Lemma~\ref{lemma:error_bound_literature}, we make use of a similar line of reasoning as used to prove Lemma~\ref{lemma:error_bound_literature} in its original source \cite{romeijnders2016general}. The main differences are twofold. Firstly, we will use the asymptotic periodicity result from Lemma~\ref{lemma:sigma_bar_k} (in which the vectors $\bar{\sigma}^k$, $k \in K^q$ do not depend on $q$) rather than the analogue result from \cite{romeijnders2016general} (in which the corresponding vectors $\sigma^k$, $k \in K^q$ do depend on $q$). Secondly, we use the upper bound on $\| v^q - \hat{v}^q \|_\infty$ from Lemma~\ref{lemma:bound_v_v_hat}, whose dependence on $q$ is expressed explicitly, rather than the analogous result from \cite{romeijnders2016general}.

\begin{lemma} \label{lemma:bound_given_q}
    Consider the recourse function $Q$ and its shifted LP-relaxation approximation $\hat{Q}$ from Definition~\ref{def:shifted_LP-relaxation} and assume that $q \in \Xi^q$ and $T \in \Xi^T$ are fixed. Then, there exists a finite constant $C > 0$, not depending on $T$ and $q$, such that for all $f \in \mathcal{H}^m$, we have
    \begin{align*}
        \| Q - \hat{Q} \|_\infty &\leq C \cdot \| q \|_1 \cdot \sum_{i=1}^m \E\big[ \totvar f_i(\cdot|h_{-i}) \big].    
    \end{align*}
\end{lemma}
\begin{proof}
    Let $x \in \R^{n_1}$ be given. Then, we have
    \begin{align*}
       \big| Q(x) - \hat{Q}(x) \big| &= \big| \int_{\R^m}  \big( v(\omega) - \hat{v}(\omega) \big) g(\omega) d \omega \big| ,
    \end{align*}
    where $g$ is the pdf satisfying $g(\omega) = f(\omega + Tx)$, $\omega \in \R^m$. We will split up this integral into several integrals over subsets of its domain $\R^m$. By Lemma~\ref{lemma:sigma_bar_k} we know that the value function approximation error $v^q - \hat{v}^q$ is $B^k$-periodic on $\bar{\sigma}^k + \Lambda^k$, $k \in K^q$. Writing $\mathcal{N} := \R^m \setminus \cup_{k \in K}( \bar{\sigma}_k + \Lambda^k)$ for the complement set, we have
    \begin{align}
        &\big| \E^h[v^q(h - Tx) - \hat{v}^q(h - Tx)] \big| \\
        &\leq \sum_{k \in K^q} \Big| \int\limits_{\bar{\sigma}_k + \Lambda^k}  \big( v^q(\omega) - \hat{v}^q(\omega) \big) g(\omega) d \omega \Big| \label{eq:term_sigma}\\
        &+  \int_{\mathcal{N}}|  v^q(\omega) - \hat{v}^q(\omega) | g(\omega) d\omega. \label{eq:term_N}
    \end{align}
    Observe that, importantly, the sets $\bar{\sigma}^k + \Lambda^k$, $k \in K^q$, and $\mathcal{N}$ do not depend on $q$ (as opposed to their analogues in \cite{romeijnders2016general}). 
    
    Now we bound both terms in the right-hand side above separately. Applying Theorem~4.13 from \cite{romeijnders2016general}, providing an upper bound on the integral of a zero-mean $B^k$-periodic function, and using the upper bound on $\| v^q - \hat{v}^q \|_\infty$ from Lemma~\ref{lemma:bound_v_v_hat}, we obtain 
    \begin{align}
        &\left| \int_{\bar{\sigma}^k + \Lambda^k} \big(  v^q(\omega) - \hat{v}^q(\omega) \big) g(\omega) d \omega \right| \nonumber\\
        &\leq \frac{1}{2} \gamma \| q \|_1  |\det(B^k)| \sum_{i=1}^m \E\big[ \totvar f_i(\cdot|h_{-i}) \big]. \label{eq:bound_first_term}
    \end{align}
    Summing over all $k \in K^q$ then yields an upper bound on \eqref{eq:term_sigma}. However, to avoid complications related to the dependence of the index set $K^q$ on $q$, we instead sum over all $k \in \bar{K}$. This yields an upper bound that only depends on $q$ through $\| q \|_1$.
    
    For \eqref{eq:term_N}, we observe that
    \begin{align*}
        \int_{\omega \in \mathcal{N}}|  v^q(\omega) - \hat{v}^q(\omega) | g(\omega) d\omega \leq \| v^q - \hat{v}^q \|_\infty \Prob\{ \omega \in \mathcal{N} \}.
    \end{align*}
    We use Lemma~\ref{lemma:bound_v_v_hat} to bound $\| v^q - \hat{v}^q \|$. Moreover, we use an analogue of equation (5.4) in \cite{romeijnders2016general} to bound the probability $\Prob\{ \omega \in \mathcal{N} \}$. However, rather than summing over $k \in K^q$, as in \cite{romeijnders2016general}, we again sum over $k \in \bar{K}$, yielding a more conservative upper bound. We obtain
    \begin{align}
        &\int_{\omega \in \mathcal{N}}|  v^q(\omega) - \hat{v}^q(\omega) | g(\omega) d\omega \nonumber \\
        &\leq \gamma \| q \|_1 \sum_{k \in \bar{K}} \sum_{j = 1}^m D_{kj} \sum_{i=1}^m \E\big[ \totvar f_i(\cdot|h_{-i}) \big]. \label{eq:bound_second_term}
    \end{align}
    Note that the constants $D_{kj}$, $k \in \bar{K}$, $j =1,\ldots,m$, do not depend on $q$, since $\mathcal{N}$ does not depend on $q$. Combining \eqref{eq:bound_first_term} and \eqref{eq:bound_second_term} yields
    \begin{align*}
        &\big| \E^h[v^q(h - Tx) - \hat{v}^q(h - Tx)] \big| \\
        &\qquad \leq C \cdot \| q \|_1 \cdot \sum_{i=1}^m \E\big[ \totvar f_i(\cdot|h_{-i}) \big],
    \end{align*}
    where $C := \gamma \sum_{k \in \bar{K}} \Big( \frac{1}{2} | \det B^k | +  \sum_{j = 1}^m D_{kj} \Big)$. The result now follows from the observation that the right-hand side above does not depend on $x$ and $T$.
\end{proof}

Lemma~\ref{lemma:bound_given_q} provides an error bound for the approximation error $\| Q - \hat{Q}\|_\infty$ under the assumption that $q$ and $T$ are fixed. Compared with the error bound from \cite{romeijnders2016general}, restated in Lemma~\ref{lemma:error_bound_literature}, the dependence of the error bound in Lemma~\ref{lemma:bound_given_q} on the second-stage cost vector $q$ is made explicit. Extending the result to a setting where $T$ is random yields a parametric error bound in $q$, which constitutes the first main result of our paper.

\begin{theorem} \label{theorem:error_bound_fixed_q}
Consider the recourse function $Q$ and its shifted LP-relaxation approximation $\hat{Q}$ from Definition~\ref{def:shifted_LP-relaxation} and assume that only $q \in \Xi^Q$ is fixed. Then, there exist finite constants $C_1, C_2 > 0$, not depending on $q$, such that for all $f \in \mathcal{H}^m$, we have
\begin{align*}
    \| Q - \hat{Q} \|_\infty &\leq C_1 \cdot \| q \|_1 \cdot \sum_{i=1}^m \E\big[ \totvar f_i(\cdot|h_{-i}) \big],
\end{align*}
and
\begin{align*}
    \| Q - \tilde{Q}_\alpha \|_\infty &\leq C_2 \cdot \| q \|_1 \cdot \sum_{i=1}^m \E\big[ \totvar f_i(\cdot|h_{-i}) \big].    
\end{align*}
\end{theorem}
\begin{proof}
    We prove the first inequality; the second inequality follows analogously. Let $x \in \R^{n_1}$ be given. Then, by \eqref{eq:approx_error_q_t_h} and Lemma~\ref{lemma:bound_given_q}, there exists $C_1>0$, not depending on $q$ and $T$, such that for every $f \in \mathcal{H}^m$,
    \begin{align*}
        | Q(x) - \hat{Q}(x) | &\leq \E^{T}\bigg[ C_1 \cdot \| q \|_1 \cdot \sum_{i=1}^m \E\big[ \totvar f_i(\cdot|h_{-i}) \big] \bigg]\\
        &= C_1 \cdot \| q \|_1 \cdot \sum_{i=1}^m \E\big[ \totvar f_i(\cdot|h_{-i}) \big].
    \end{align*}
    The result follows from the fact that the right-hand side above does not depend on the value of $x$.
\end{proof}

Theorem~\ref{theorem:error_bound_fixed_q} provides parametric error bounds that explicitly depend on the second-stage cost vector $q$. We find that the error bound grows linearly in the $\ell_1$-norm of $q$, which we might interpret as a measure of the ``magnitude'' of $q$. Interestingly, this results generalizes the result in Theorem~5 of \cite{kleinhaneveld2006simple}, which provides an error bound for $\alpha$-approximations of simple integer recourse models that scale linearly in the sum of the elements of $q$, which are assumed to be non-negative in that paper. While the result in \cite{kleinhaneveld2006simple} only applies to the very special case of simple integer recourse, our Theorem~\ref{theorem:error_bound_fixed_q} holds for much more general models.

Finally, we extend Theorem~\ref{theorem:error_bound_fixed_q} to a setting where $q$ is random. This yields the second main result of our paper.

\begin{theorem}\label{theorem:error_bound}
    Consider the recourse function $Q$ from \eqref{eq:model} and its convex approximations $\hat{Q}$ and $\tilde{Q}_\alpha$ from Definition~\ref{def:shifted_LP-relaxation} and \ref{def:alpha-approximation}, respectively, and suppose that Assumption~\ref{ass:main} holds. Then, there exist constants $C_1, C_2 >0$ such that for every $f \in \mathcal{H}^m$, we have
    \begin{align*}
        \| Q - \hat{Q} \|_\infty \leq C_1 \cdot \E\big[ \| q \|_1 \big] \sum_{i=1}^m \E\big[ \totvar f_i(\cdot | h_{-i}) \big],
    \end{align*}
    and
    \begin{align*}
        \| Q - \tilde{Q}_\alpha \|_\infty \leq C_2 \cdot \E\big[ \| q \|_1 \big] \sum_{i=1}^m \E\big[ \totvar f_i(\cdot | h_{-i}) \big].
    \end{align*}
\end{theorem}
\begin{proof}
    We prove the first inequality; the second inequality follows analogously. Let $x \in \R^{n_1}$ be given. Then, by applying Lemma~\ref{lemma:bound_given_q} to \eqref{eq:approx_error_q_t_h} we find that there exists $C_1 >0$ such that for every $f \in \mathcal{H}^m$,
    \begin{align*}
        \big| Q(x) - \hat{Q}(x) \big| &\leq \E^{q,T}\bigg[ C_1 \cdot \| q \|_1 \cdot \sum_{i=1}^m \E\big[ \totvar f_i(\cdot|h_{-i}) \big] \bigg] \\
        &= C_1 \cdot \E \big[ \| q \|_1 \big] \cdot \sum_{i=1}^m \E\big[ \totvar f_i(\cdot|h_{-i}) \big].
    \end{align*}
    The result now follows from the observation that the right-hand side above does not depend on the value of $x$.
\end{proof}

Theorem~\ref{theorem:error_bound} provides error bounds that explicitly depend on the distribution of $q$. The error bounds can be represented as the product of two non-negative factors: one depending on the distribution of $q$ and another depending on the distribution of $h$. The first factor, $\E \big[ \| q \|_1 \big]$, captures the dependence of the error bound on the distribution of $q$. Following the discussion above, we might interpret this as the average ``magnitude'' of $q$. It shows that our error bound is indeed finite if $\E \big[ \| q \|_1 \big] < +\infty$, i.e., under Assumption~\ref{ass:main}\eqref{ass:finite_exp}. Note that besides this assumption, no other assumptions about the distribution of $q$ are made. In this regard, we improve upon the error bounds from Theorem~\ref{theorem:error_bound} in \cite{vanbeesten2020convex}, which only hold if $q$ is discretely distributed on a finite support. In particular, our error bounds can also deal with, e.g., continuously distributed $q$.

The second factor, related to the distribution of $h$, is of the same form as error bounds from the literature. It depends on the total variations of the one-dimensional conditional density functions of the random right-hand side vector $h$. It converges to zero as these total variations go to zero. Practically speaking, this means that our convex approximations are good if the dispersion in the distribution of $h$ is large. Another way of interpreting this is that a highly dispersed distribution of $h$ leads to a ``near-convex'' model. Interestingly, a similar ``convexification'' effect is not observed in terms of the distribution of $q$: the dispersion of this distribution does not affect the error bounds. Only the average ``magnitude'' $\E\big[\| q \|_1 \big]$ matters.

\section{Conclusion} \label{sec:conclusion}

We consider performance guarantees for convex approximations of MIR models in the form of error bounds: upper bounds on the approximation error. In contrast with the literature, we derive \textit{parametric} error bounds that explicitly depend on the second-stage cost vector $q$ or its distribution, in case $q$ is random. We consider two particular convex approximations from the literature: the shifted LP-relaxation approximation and the $\alpha$-approximation and for each approximation we derive a corresponding error bound. 

Using properties of the value function approximation error, we first derive an error bound that holds when $q$ is fixed. Although such error bounds exist in the literature, our error bound is special in the sense that its dependence on the second-stage cost vector $q$ is made explicit: the bound scales linearly in the $\ell_1$ norm of $q$. We might interpret this scaling factor as the ``magnitude'' of $q$. Next, we use this bound to derive an error bound that holds when $q$ is random. The error bound scales linearly in the expected value $\E \big[ \| q \|_1 \big]$, which we might interpret as the \textit{average} ``magnitude'' of $q$. Hence, our convex approximations are good if this expected value is small.

Future research may be aimed at deriving error bounds under even more relaxed distributional assumptions in the MIR model. For instance, it would be interesting to investigate the case where not all elements of $h$ are random, or where some elements of $h$ are fully dependent.


\bibliography{references}

\begin{thebibliography}{10}

\bibitem{Ahmed2004}
{\sc Ahmed, S., Tawarmalani, M., and Sahinidis, N.~V.}
\newblock {A finite branch-and-bound algorithm for two-stage stochastic integer
  programs}.
\newblock {\em Mathematical Programming 100\/} (2004), 355--377.

\bibitem{Caroe1999}
{\sc Car{\o}e, C.~C., and Schultz, R.}
\newblock {Dual decomposition in stochastic integer programming}.
\newblock {\em Operations Research Letters 24\/} (1999), 37--45.

\bibitem{cook1986}
{\sc Cook, W., Gerards, A. M.~H., Schrijver, A., and Tardos, {\'E}.}
\newblock Sensitivity theorems in integer linear programming.
\newblock {\em Mathematical Programming 34}, 3 (1986), 251--264.

\bibitem{kleinhaneveld2006simple}
{\sc Klein~Haneveld, W.~K., Stougie, L., and {van der Vlerk}, M.~H.}
\newblock Simple integer recourse models: convexity and convex approximations.
\newblock {\em Mathematical Programming 108}, 2--3 (2006), 435--473.

\bibitem{kleinhaneveld1999}
{\sc Klein~Haneveld, W.~K., and van~der Vlerk, M.~H.}
\newblock Stochastic integer programming: General models and algorithms.
\newblock {\em Annals of Operations Research 85\/} (1999), 39--57.

\bibitem{Laporte1993}
{\sc Laporte, G., and Louveaux, F.~V.}
\newblock {The integer L-shaped method for stochastic integer programs with
  complete recourse}.
\newblock {\em Operations Research Letters 13\/} (1993), 133--142.

\bibitem{ntaimo2010disjunctive}
{\sc Ntaimo, L.}
\newblock Disjunctive decomposition for two-stage stochastic mixed-binary
  programs with random recourse.
\newblock {\em Operations Research 58}, 1 (2010), 229--243.

\bibitem{romeijnders2017assessing}
{\sc Romeijnders, W., Morton, D.~P., and {van der Vlerk}, M.~H.}
\newblock Assessing the quality of convex approximations for two-stage totally
  unimodular integer recourse models.
\newblock {\em INFORMS Journal on Computing 29}, 2 (2017), 211--231.

\bibitem{romeijnders2016general}
{\sc Romeijnders, W., Schultz, R., {van der Vlerk}, M.~H., and Klein~Haneveld,
  W.~K.}
\newblock A convex approximation for two-stage mixed-integer recourse models
  with a uniform error bound.
\newblock {\em SIAM Journal on Optimization 26}, 1 (2016), 426--447.

\bibitem{romeijnders2015}
{\sc Romeijnders, W., {van der Vlerk}, M.~H., and Klein~Haneveld, W.~K.}
\newblock Convex approximations for totally unimodular integer recourse models:
  A uniform error bound.
\newblock {\em SIAM Journal on Optimization 25}, 1 (2015), 130--158.

\bibitem{romeijnders2016tu}
{\sc Romeijnders, W., {van der Vlerk}, M.~H., and Klein~Haneveld, W.~K.}
\newblock Total variation bounds on the expectation of periodic functions with
  applications to recourse approximations.
\newblock {\em Mathematical Programming 157}, 1 (2016), 3--46.

\bibitem{schultz2003stochastic}
{\sc Schultz, R.}
\newblock Stochastic programming with integer variables.
\newblock {\em Mathematical Programming 97}, 1 (2003), 285--309.

\bibitem{schultz1998solving}
{\sc Schultz, R., Stougie, L., and Van Der~Vlerk, M.~H.}
\newblock Solving stochastic programs with integer recourse by enumeration: A
  framework using gr{\"o}bner basis.
\newblock {\em Mathematical Programming 83}, 1 (1998), 229--252.

\bibitem{sen2005algorithms}
{\sc Sen, S.}
\newblock Algorithms for stochastic mixed-integer programming models.
\newblock {\em Handbooks in Operations Research and Management Science 12\/}
  (2005), 515--558.

\bibitem{Sen2005}
{\sc Sen, S., and Higle, J.~L.}
\newblock {The C3 Theorem and a D2 Algorithm for Large Scale Stochastic
  Mixed-Integer Programming: Set Convexification}.
\newblock {\em Mathematical Programming 104\/} (2005), 1--20.

\bibitem{vanbeesten2020convex}
{\sc van Beesten, E.~R., and Romeijnders, W.}
\newblock Convex approximations for two-stage mixed-integer mean-risk recourse
  models with conditional value-at-risk.
\newblock {\em Mathematical Programming 181\/} (2020), 473--507.

\bibitem{vanderlaan2020LBDA}
{\sc van~der Laan, N., and Romeijnders, W.}
\newblock A loose {B}enders decomposition algorithm for approximating two-stage
  mixed-integer recourse models.
\newblock {\em Mathematical Programming\/} (2020), 1--34.

\bibitem{vanderlaan2018higher}
{\sc van~der Laan, N., Romeijnders, W., and van~der Vlerk, M.~H.}
\newblock Higher-order total variation bounds for expectations of periodic
  functions and simple integer recourse approximations.
\newblock {\em Computational Management Science 3}, 15 (2018), 325--349.

\bibitem{vlerk2004}
{\sc {van der Vlerk}, M.~H.}
\newblock Convex approximations for complete integer recourse models.
\newblock {\em Mathematical Programming 99}, 2 (2004), 297--310.

\bibitem{vlerk2010}
{\sc {van der Vlerk}, M.~H.}
\newblock Convex approximations for a class of mixed-integer recourse models.
\newblock {\em Annals of Operations Research 177}, 1 (2010), 139--150.

\end{thebibliography}

\end{document}